\documentclass[reqno,11pt]{amsart}
\usepackage{amsmath,amsthm,amssymb,amsfonts,epsfig,color,graphicx,enumerate,accents,subfigure}
\usepackage[a4paper,margin=3truecm]{geometry}
\usepackage{tikz}

\DeclareMathOperator{\dive}{div}

\DeclareMathOperator{\dist}{dist}

\def\ds{\displaystyle}
\def\eps{{\varepsilon}}
\def\N{\mathbb{N}}
\def\R{\mathbb{R}}
\def\O{\Omega}

\def\F{\mathcal{F}}
\def\HH{\mathcal{H}}

\newcommand{\be}{\begin{equation}}
\newcommand{\ee}{\end{equation}}

\newcommand{\ep}{\varepsilon}

\numberwithin{equation}{section}
\theoremstyle{plain}
\newtheorem{teo}{Theorem}[section]
\newtheorem{lemma}[teo]{Lemma}

\newtheorem{prop}[teo]{Proposition}

\theoremstyle{remark}

\newenvironment{ack}{{\bf Acknowledgements. }}

\title[]{Sharp estimates for the anisotropic torsional rigidity and the principal frequency}

\author{Giuseppe Buttazzo}
\address{Dipartimento di Matematica, Universit\`a di Pisa, Largo B. Pontecorvo 5, 56126 Pisa, ITALY}
\email{buttazzo@dm.unipi.it}

\author{Serena Guarino Lo Bianco}
\address{Dipartimento di Matematica e Applicazioni� Renato Caccioppoli� Universit\`a degli Studi di Napoli Federico II,Via Cintia, Monte S. Angelo,80126 Napoli, ITALY}
\email{serena.guarinolobianco@unina.it}

\author{Michele Marini}
\address{Scuola Internazionale Superiore di Studi Avanzati, via Bonomea 265, 34136 Trieste, ITALY}
\email{mmarini@sissa.it}

\begin{document}

\maketitle
\begin{abstract}
In this paper we generalize some classical estimates involving the torsional rigidity and the principal frequency of a convex domain to a class of functionals related to some anisotropic non linear operators.
\end{abstract}



\section{Introduction}\label{sintro}

Let $h_K$ be the norm associated to a convex body $K$ (see Section \ref{prel} for more details); given a domain $\O\subset\R^N$ with finite measure, we define the $K$-principal frequency, $\lambda_1^K$, and the $K$-torsional rigidity, $T^K$, as 
\be\label{lambda1}
\lambda_1^K(\O)=\min_{u\in W_0^{1,2}(\O)\setminus\{0\}}\frac{\int_\O h^2_K(\nabla u)\,dx}{\int_\O u^2\,dx}\;,
\ee
and
\be\label{torsion}
T^K(\O)=\max_{u\in W_0^{1,2}(\O)\setminus\{0\}}\frac{\Big(\int_\O u\,dx\Big)^2}{\int_\O h^2_K(\nabla u)\,dx}\;.
\ee
It is convenient to introduce the function $H_K=h^2_K/2$; when $H_K$ is sufficiently smooth, we can write the {\it Euler-Lagrange equations} for the minimizers of the problems \eqref{lambda1} and \eqref{torsion} to get a PDE interpretation of the above quantities. Indeed, the $K$-principal frequency is related to the eigenvalue problem 
\be\label{lambdapde}
-\Delta_Ku=\lambda_1^Ku\mbox{ in }\O,\qquad u=0\mbox{ on }\partial\O,
\ee
while the $K$-torsional rigidity is the $L^1$ norm of the solution $u$ of:
\be\label{torsionpde}
-\Delta_Ku=1\mbox{ in }\O,\qquad u=0\mbox{ on }\partial\O.
\ee
Here $\Delta_K$ denotes the {\it Finsler-Laplace operator} given by
\be\label{Klap}
\Delta_K u = \dive (DH_K(\nabla u)).
\ee
In the Euclidean case, occurring when $K$ is the unitary ball $B$, (and $h_K(x)=|x|$) the operator given in \eqref{Klap} coincides with the Laplacian and $\lambda_1$ and $T$ are the usual first Dirichlet eigenvalue and torsional rigidity.

As in the linear case, the quantities defined in \eqref{lambda1} and \eqref{torsion} are monotone, in opposite sense, with respect to the set inclusion, {\it i.e.} if $\O_1\subset\O_2$ then
\be\label{monotonicity}
\lambda_1^K(\O_1)\ge\lambda_1^K(\O_2)\quad \mbox{and}\quad T^K(\O_1)\le T^K(\O_2).
\ee
Moreover, since $H_K$ is a homogeneous function of degree 2, the following scalings hold true:
\be\label{scaling}
\lambda_1^K(t\O)=t^{-2}\lambda_1^K(\O)\qquad\mbox{and}\qquad T^K(t\O)=t^{N+2}T^K(\O),\quad t>0.
\ee

Shape optimization problems involving $\lambda_1$ and $T$, or even more general spectral functionals of the form $\F(\O)=\Phi(\lambda_1(\O),T(\O))$, are widely studied in the literature (see for instance \cite{bfnt}, \cite{bra14}, \cite{bucbut}, \cite{but}, \cite{buvdbve}, \cite{hen}, \cite{henpie}) and, as it is well known, it is possible to get both lower and upper bounds for the principal frequency and the torsional rigidity in terms of quantities associated to the geometry of the domain $\O$, such as, for instance the perimeter and the volume (just think to the {\it Faber-Krahn inequality} and the {\it Saint-Venant theorem}, see for instance the recent book \cite{hen}).

As it should not be unexpected, if we impose some further constraints in the class of admissible domains, we can get stronger estimates. The class of convex domains, for instance, has been considered by several authors: on one hand the {\it a priori} assumption of the convexity of the domain naturally arises in many situations; on the other, the class of convex sets has strong compactness properties which ensure the existence of extremal domains for a great number of geometric inequalities.

In this paper we are interested in estimates of the principal frequency and the torsional rigidity of a convex domain in terms of the {\it inradius}, $R_\O$, {\it i.e.} the radius of the biggest ball contained in $\O$.

An immediate consequence of \eqref{monotonicity} and \eqref{scaling} is that, for the Euclidean case
\be\label{easyest}
\lambda_1(\O)\le \lambda_1(B_{R_\O})=\lambda_1(B_1) R_\O^{-2}.
\ee
A classical result by J. Hersch (see \cite{her}) shows that for any convex domain $\O\subset\R^2$ it holds
\be\label{hersch}
\frac{\pi^2}{4}R_\O^{-2}\le\lambda_1(\O),
\ee
and the inequality is sharp: if we allow unbounded domains, equality case occurs when $\O$ is a strip, otherwise it is reached only asymptotically, by a sequence of rectangles with sides $a\ll b$. Hersch's technique has been extended to convex domains of $\R^N$ by M. H. Protter in \cite{pro} who proved the validity of \eqref{hersch} in every dimension.

Concerning the torsional rigidity, in \cite{polsze} Polya and Szego proved that, for any domain $\O\subset\R^2$, the following inequality holds true
\be\label{p-s}
\frac{T(\O)}{|\O|} \ge \frac{1}{8}R_\O^2,
\ee
where $|\O|$, denotes the {\it Lebesgue measure} of $\O$.
Again, this inequality is sharp, becoming an equality if $\O$ is a ball.

An upper bound for the torsional rigidity was obtained by E. Makai in \cite{mak}, who proved that, for every convex domain $\O\subset\R^2$, it holds true that
\be\label{makeq}
\frac{T(\O)}{|\O|} \le \frac{1}{3}R_\O^2.
\ee
Inequality \eqref{makeq} is also sharp, and the best constant is achieved, again, if we consider a sequence of rectangles with sides $a\ll b$.

The aim of this paper is to extend to the anisotropic case and to a general dimension the estimates \eqref{easyest}, \eqref{hersch}, \eqref{p-s} and \eqref{makeq}. By virtue of \eqref{monotonicity}, and \eqref{scaling} it seems reasonable to find estimates of the form
$$c(N,K)(R_\O^K)^{-2}\le \lambda_1^K(\O)\le C(N,K)(R_\O^K)^{-2}$$
and
$$\overline c(N,K)(R_\O^K)^{2}\le \frac{T^K(\O)}{|\O|}\le \overline C(N,K)(R_\O^K)^{2},$$
where $R_\O^K$ is the {\it anisotropic inradius}, {\it i.e.} the largest number $t$ such that $x+tK\subseteq\O$, for some $x\in\O$. We show that, surprisingly, for the best constants in the above formulas, there is no dependence on $K$, this means that the bounds which are in force in the Euclidean case, hold true for every choice of the nonlinear operator $\Delta_K$.

In the following section, we illustrate more precisely all the results that we prove throughout this paper; here we limit ourselves to stress that, besides their own interest, the proofs of such results may provide a more geometrical insight for the special case $K=B$ as well, and may help to shed light on what are the most relevant assumptions that we need to impose on a non linear operator, in order to expect those kind of estimates.

\section{Main results}

In Section \ref{torsione}, we extend to the anisotropic case inequalities \eqref{p-s} and \eqref{makeq}. More precisely,  in Theorem \ref{upt}, we show that, for any convex domain $\O\subset\R^N$, and any 2-homogeneous $C_+^2$-regular\footnote{See Section \ref{prel} for the definition.} function, $H_K$, the following {\it a priori} bounds on the $K$-torsional rigidity holds true

\be\label{torsion1}
\frac{1}{N(N+2)}(R_\O^K)^2\le\frac{T^K(\O)}{|\O|}\le \frac{1}{3}(R_\O^K)^2.
\ee

As for the corresponding linear case, equality can be achieved in the first inequality of \eqref{torsion1} when $\O$ coincides with $K$ up to translations and dilations, while (asymptotic) equality in the second inequality can be obtained by considering a sequence of rectangles with sides $a\ll b$, see Proposition \ref{rect} (notice that $1/3$ is the best constant in every dimension).

As far as we know, in the case $K=B$, even if the proof is significantly easier, there is not a specific reference for estimates of the form \eqref{torsion1} in general dimension available in the literature.

The proofs of both Theorems \ref{lowtors} and \ref{upt} are based on the choice of a suitable {\it one dimensional test function}. Since we defined the torsional as a supremum, it is easy to get an estimate from below using the variational formulation, while to get a bound from above it is important to use the PDE interpretation explained in \eqref{torsionpde}.

In Section \ref{eigenvalue}, we extend the result by Hersch and Protter by proving the formula 
\be\label{eigen}
\frac{\pi^2}{4}(R_\O^K)^{-2}\le \lambda_1^K(\O)\le \lambda_1(B) (R_\O^K)^{-2}.
\ee

As in the linear case the first inequality holds true as an equality if we consider a strip (or a sequence of ``thin'' rectangles, if we do not allow unbounded domains). The case of equality in the second inequality occurs when we choose $\O=x+rK$, for some $x\in\R^N$ and $r>0$; as in the linear case, the latter estimate follows by virtue of the monotonicity and the scaling laws. We prove in Proposition \ref{lambdaball} that $\lambda_1^K(K)=\lambda_1(B)$.

In Section \ref{reg}, we remove any regularity assumption on the convex body $K$. This result easily follows by approximating $K$ with smooth convex bodies $K_n$ and then by passing to the limit the corresponding inequalities for the sequence $K_n$.

We conclude this section by remarking that, in a recent paper \cite{dpigav} authors give a sharp estimate for the anisotropic principal frequency and torsional rigidity\footnote{They consider a slightly more general class of operators since they allow, in our notations, functions $H_K$ which are homogeneous of degree $p$.} in terms of the anisotropic perimeter and Lebesgue measure. We point out that thin rectangles are limit sets for the anisotropic torsional rigidity as well; but curiously, while such a sequence maximizes the ratio between torsional rigidity and volume among all the sets of given anisotropic inradius, it minimizes the same ratio among all sets of given anisotropic perimeter.

\section{Preliminaries}\label{prel}

In this section, we recall some basic notions concerning convex bodies and anisotropic differential operators.

Associated with a convex body $K\subset\R^N$, there is the {\it support function}, $h_K: \R^N\to\R$ defined as
$$h_K(x)=\max\big\{x\cdot y\ :\ y\in K\big\}.$$
The support function is a convex and positive 1-homogeneous function, when $K$ is symmetric with respect to the origin, $h_K$ is actually a norm, whose unit ball, $K^*=\{x\ :\ h_K(x)\le1\}$, is called {\it polar body} of $K$.

We denote by $d_K$ the anisotropic distance induced by $K$. In particular, given a set $A$ and a point $x$ we denote
\be\label{edist}
\dist_K(x,A)=\inf\big\{h_{K^*}(a-x)\ :\ a\in A\big\}.
\ee
It is useful to define the {\it anisotropic inradius} in terms of such a distance as
$$R_\O^K=\sup\big\{\dist_K(x,\partial\O)\ :\ x\in\O\big\}.$$

We now recall some duality relations between $K$ and $K^*$ that we use in the following sections and we refer to \cite[Chapter 1]{Scn} for a more comprehensive account on this subject.

$K^{**}=K$, namely $^*$ is an involution in the set of convex bodies; this means that in all the following relations the role played by $K$ and $K^*$ can be interchanged. In particular $h_{K^*}$ is a norm as well, $K=\big\{x\ :\ h_{K^*}(x)\le1\big\}$ is the corresponding unit ball and $(\R^N,h_{K^*})$, as a Banach space, is the dual space of, $(\R^N, h_K)$.

$H_K^*=h_{K^*}^2/2$ is the {\it Legendre-Fenchel transformation} of $H_K=h_K^2/2$; this entails that, when both the functions above are differentiable, the gradient mappings $DH_K$ and $DH_{K}^*$ are one the inverse of the other, namely
\be\label{invgrad}
DH_K(DH_K^*(x))=x\qquad\hbox{for every }x\in\R^N.
\ee

A complete description of the differential theory for convex bodies and support functions can be found in \cite{Scn}, here we limit ourselves to recall that the obstruction to the regularity of $h_K$ is the possible presence of ``flat'' parts in the boundary of $K$; in particular, if $K$ is strictly convex, then $K^*$ (and thence $h_K$ ) is differentiable and 
\be\label{gradsupp}
Dh_K(x)=y,
\ee
where $y$ is the only point in $\partial K$ such that the outer normal unit to the boundary of $K$ at $y$ is $x/|x|$, in other words, $Dh_K$ is the $0$-homogeneous extension of the inverse function of the {\it Gauss map}. Conversely, if $K$ has differentiable boundary, then $h_{K^*}$ is differentiable and $K^*$ is strictly convex. Moreover, for $x\in\partial K$
\be\label{gradnorm}
Dh_{K^*}(x)/|Dh_{K^*}(x)|=\nu_K(x),
\ee
where $\nu_K(x)$ is the outer unit normal to $\partial K$ at $x$.

A straightforward consequence of \eqref{gradsupp} is the following important formula:
\be\label{hdh}
h_{K^*}(Dh_K(x))=1\qquad\hbox{for every }x\ne0.
\ee

A similar criterion holds true for higher order differentiability of $H_K$: indeed, when $K$ is $C_+^2$-regular, {\it i.e.} the boundary of $K$ can be written as the graph of a twice differentiable function with positive Hessian, then both $H_K$ and $H_K^*$ are $C^2$-regular, and their Hessian matrices are positive. By virtue of the considerations above the differential operator $\Delta_K$, acting on $C^2$ functions as
$$\Delta_Ku(x)=\dive(DH_K(\nabla u(x))),$$
is uniformly elliptic, and classical existence and regularity results for the solutions of problems \eqref{lambdapde} and \eqref{torsionpde} can be applied.

\section{$K$-torsional rigidity}\label{torsione}

Throughout this section we assume that $K$ is a $C_+^2$-regular symmetric convex body, so that its support function, $h_K$, is a smooth norm on $\R^N$. We are going to show that balls (with respect to the metric induced by $h_K$) are extremal bodies for \eqref{torsion1}; in the following lemma we explicitly compute the $K$-torsional rigidity for those sets. As we shall see, this quantity does not depend on the choice of the norm.

\begin{lemma}\label{torsball}
Let $T^K$ be as in \eqref{torsion}, then
$$\frac{T^K(rK)}{|rK|}=\frac{r^2}{N(N+2)}\;.$$
\end{lemma}

\begin{proof}
We start by constructing the solution of the problem \eqref{torsionpde}. From \eqref{invgrad} we deduce
$$\Delta_K H_K^*(x)=\dive\left[DH_K(DH_K^*(x))\right]=\dive x=N.$$
Then the function $u=(r^2-h^2_{K^*})/2N$ is a solution of
$$-\Delta_K u=1\mbox{ in }rK,\qquad u=0\mbox{ on }\partial(rK).$$
We have
\[\begin{split}
\int_{rK}h^2_{K^*}\,dx&=\int_0^r\int_{\{h_{K^*}=t\}}\frac{h^2_{K^*}}{|Dh_{K^*}|}d\HH^{N-1}\,dt\\
&=\int_0^r t^2\int_{\partial(tK)}\frac{1}{|Dh_{K^*}|}\,d\HH^{N-1}\,dt\\
&=\int_0^r t^{N+1}\int_{\partial K}\frac{1}{|Dh_{K^*}|}\,d\HH^{N-1}\,dt=\frac{r^{N+2}}{N+2}\int_{\partial K}\frac{1}{|Dh_{K^*}|}\,d\HH^{N-1}\;,
\end{split}\]
where we used the co-area and the change of variable formulas. Let us evaluate the last integral. Thanks to \eqref{gradnorm} and \eqref{hdh}, and since the support function is 1-homogeneous, we get
\[\begin{split}
\int_{\partial K}\frac{1}{|Dh_{K^*}|}\,d\HH^{N-1}&=\int_{\partial K}h_k\Big(\frac{Dh_{K^*}}{|Dh_{K^*}|}\Big)\,d\HH^{N-1}\\
&=\int_{\partial K}h_K(\nu_K(x))\,d\HH^{N-1}\\
&=\int_{\partial K}x\cdot \nu_K(x)\,d\HH^{N-1}\\
&=\int_K\dive x\,dx=N|K|.
\end{split}\]
Therefore we obtain
$$T^K(rK)=\int_{rK}\frac{r^2-h_{K^*}}{2N}\,dx=\frac{1}{2N}\Big(r^2|rK|-\frac{r^{N+2}}{N+2}N|K|\Big)=\frac{r^2|rK|}{N(N+2)}$$
as required.
\end{proof}

In the following proposition we prove a lower bound for the $K$-torsional rigidity for strictly convex domains. In such a case, the function
\be\label{ubar}
\overline u(x)=\frac{1-\big(h_{\O^*}(x)\big)^2}{2N}
\ee
used in the above lemma to compute the $K$-torsional rigidity of balls is no longer the solution of the anisotropic torsion problem, but it still vanishes on the boundary of the domain, and it can be used as a test function. Notice that, since $T^K$ maximizes the quotient
$$\left(\int_\O u\,dx\right)^2\left(\int_\O h^2_K(\nabla u)\,dx\right)^{-1}\;,$$
then every test function provides a lower bound for the torsion. Our choice of the test function $\overline u$ is motivated by the fact that it provides the optimal lower bound, as stated in Theorem \ref{lowtors} below.

\begin{teo}\label{lowtors}
Let $\O\subset \R^N$ be a convex bounded domain, then
\be\label{torslow}
\frac{T^K(\O)}{|\O|}\ge \frac{1}{N(N+2)}(R_\O^K)^2.
\ee
Moreover \eqref{torslow} is sharp, and equality occurs only if $\O=x+rK$, for some $x\in\R^N$ and $r\ge 0$.
\end{teo}

\begin{proof}
We first prove inequality \eqref{torslow} under the assumption that $\O$ is strictly convex.
Up to a translation of the domain, we can always assume that $R_\O^K K\subseteq\O$, so that $h_\O\ge R_\O^K h_K$. Let $\overline{u}$ be the function in \eqref{ubar}; $\overline u$ is an admissible function for the problem \eqref{torsion}, hence
$$T^K(\O)=\max_{u\in W_0^{1,2}(\O)\setminus\{0\}}\frac{\Big(\int_\O u\,dx\Big)^2}{\int_\O h^2_K(\nabla u)\,dx}\ge\frac{\Big(\int_\O\overline{u}\,dx\Big)^2}{\int_\O h^2_K(\nabla\overline{u})\,dx}\;.$$
Since $\O$ is a strictly convex domain, the function $h_{\O^*}$ is differentiable and, by repeating similar computations as those in Lemma \ref{torsball}, we have
\[\begin{split}
\int_\O h^2_{\O^*}\,dx
&=\int_0^1\int_{\{h^2_{\O^*}=t\}}\frac{t}{|Dh^2_{\O^*}|}\,d\HH^{N-1}\,dt\\
&=\int_0^1\frac t2\int_{\{h^2_{\O^*}=t\}}\frac{h_\O(Dh_{\O^*})}{h_{\O^*}|Dh_{\O^*}|}\,d\HH^{N-1}\,dt\\
&=\int_0^1\frac{t^{N/2}}{2}\int_{\partial\O}h_\O(\nu_\O(x))\,d\HH^{N-1}\,dt\\
&=\frac{1}{N+2}\int_{\partial\O}x\cdot\nu_\O(x)\,d\HH^{N-1}\\
&=\frac{1}{N+2}\int_\O\dive x\,dx=\frac{N|\O|}{N+2}\;.
\end{split}\]
Hence 
\be\label{num}
\left(\int_\O\overline{u}\,dx\right)^2=\frac{|\O|^2}{N^2(N+2)^2}\;.
\ee

We now compute the denominator. Again, the co-area and the change of variable formulas give
\be\label{den1}
\begin{split}
\int_\O h^2_K(\nabla\overline{u})\,dx
&=\frac{1}{4N^2}\int_\O h^2_K(Dh^2_{\O^*})\,dx\\
&=\frac{1}{4N^2}\int_0^1\int_{\{h^2_{\O^*}=t\}}\frac{h^2_K(Dh^2_{\O^*})}{|Dh^2_{\O^*}|}\,d\HH^{N-1}\,dt\\
&=\frac{1}{4N^2}\int_0^1 t^{(N-1)/2}\int_{\partial\O}\frac{h^2_K(Dh^2_{\O^*})}{|Dh^2_{\O^*}|}\,d\HH^{N-1}\,dt.
\end{split}
\ee
By using the homogeneity of the support functions and by recalling equations \eqref{hdh} and \eqref{gradnorm}, we have
\be\label{den2}
\begin{split}
\frac{h^2_K(Dh^2_{\O^*}(x))}{|Dh^2_{\O^*}(x)|}
&=\frac{|Dh^2_{\O^*}(x)|}{h_\O(Dh_{\O^*}(x))}\cdot\frac{h^2_K(Dh^2_{\O^*}(x))}{|Dh^2_{\O^*}(x)|^2}\\
&=\frac{2h_{\O^*}(x)}{h_\O(\nu_\O(x))}\cdot h_K^2(\nu_\O(x))=\frac{2t^{1/2}h_K^2(\nu_\O(x))}{h_\O(\nu_\O(x))}\;.
\end{split}
\ee
Plugging \eqref{den2} into \eqref{den1} we get
\be\label{den3}
\begin{split}
\int_\O h^2_K(\nabla\overline{u})\,dx
&=\frac{1}{4N^2}\int_0^1 t^{N/2}\int_{\partial\O}2h_\O(\nu_\O(x)) \frac{h^2_K(\nu_\O(x))}{h^2_\O(\nu_\O(x))}\,d\HH^{N-1}(x)\,dt\\
&\le\frac{1}{4N^2}\int_0^1 t^{N/2}\int_{\partial \O}2\frac{h_\O(\nu_\O(x))}{(R_\O^K)^2}\,d\HH^{N-1}(x)\,dt=\frac{1}{N(N+2)}\frac{|\O|}{(R_\O^K)^2}\;,
\end{split}
\ee
where we used the fact that $h_\O\ge R_\O^K h_K$. Combining \eqref{num} and \eqref{den3} we find:
\[
\frac{T^K(\O)}{|\O|}\ge \frac{1}{N(N+2)}(R_\O^K)^2, 
\]
as required.\\

We are now left to show the validity of \eqref{torslow} without the assumption on the strict convexity of the domain $\O$.

We recall that strictly convex bodies are a dense subset of the set of convex bodies with respect to the topology induced by the Hausdorff distance. In particular (see for instance \cite{Sc}) there exists a sequence of convex bodies $\O_n\subset\O$ such that $h_{\O_n}\llcorner\mathbb S^{N-1}$ converges to $h_\O\llcorner\mathbb S^{N-1}$ in the $C^0$-norm. Such a convergence ensures that as $n\to\infty$
\be\label{converg}
|\O_n|\to|\O|\qquad\hbox{and}\qquad R_{\O_n}^K\to R_\O^K.
\ee
From \eqref{monotonicity}, it follows that
$$T(\O)\ge T(\O_n),$$
and by applying \eqref{torslow} to each $\O_n$ equation \eqref{torslow}, we find
$$T^K(\O)\ge \frac{{|\O_n|}}{N(N+2)}(R_{\O_n}^K)^2,$$
which, combined with \eqref{converg}, gives the desired result.\\
Notice that, by virtue of Lemma \ref{torsball}, inequality \eqref{torslow} is sharp; moreover, if it holds true as an identity, then \eqref{den3} as well must be an equality, in particular
\[
\int_{\partial\O}h_\O(\nu_\O(x)) \frac{h^2_K(\nu_\O(x))}{h^2_\O(\nu_\O(x))}\,d\HH^{N-1}(x)=\int_{\partial \O}\frac{h_\O(\nu_\O(x))}{(R_\O^K)^2}\,d\HH^{N-1}(x).
\]
Since the integrand function in the left-hand side is pointwise lower than the one in the right-hand side, and since they are continuous functions, they must coincide everywhere. Namely $h_\O=R_\O^K h_K$, that is $\O=R_\O^K K$.

\end{proof}

In the following theorem we extend Makai's result \eqref{makeq} to a general dimension and to the anisotropic case. The proof given in \cite{mak} is mainly based on a clever use of the Schwarz inequality, on an inequality involving real functions and their derivatives, and on the inequality
\be\label{trivial}
|v|^2\ge v_i^2.
\ee

The second ingredient is applied to a function which is a parametrization of the boundary of the domain (we recall that in Makai's setting the boundary is a set of ``dimension one''), while \eqref{trivial} is a trivial inequality for the Euclidean norm, but is false for a generic norm. Notice that the geometric feature of the sphere that ensures the validity of \eqref{trivial} is the fact that each radius connecting the center to a boundary point is perpendicular to the tangent space at that point.

In our proof we have to change strategy: as in Makais's proof we divide $\O$ into suitable subdomains; then we construct, for each subdomain, a family of one dimensional obstacle functions, and finally we use some linear transformations to get an analogue of \eqref{trivial}.

\begin{teo}\label{upt}
Let $\O\subset\R^N$ be a bounded convex body, then
\be\label{upboundtors}
\frac{T^K(\O)}{|\O|}\le\frac13\left(R_\O^K\right)^2.
\ee
\end{teo}

\begin{proof}
We prove the claim when $\O$ is a polytope, and then the validity of \eqref{upboundtors} follows by approximation. Let us denote by $F_1,\ldots,F_n$ the facets of $\O$, and divide $\O$ into subdomains $\O_j$ defined as
$$\O_j=\big\{x\in\O\ :\ d_K(x,F_j)\le d_K(x,F_l)\mbox{ for every }l\ne j\big\},$$
where $d_K(x,F_j)$ is the distance function defined in \eqref{edist}. For every $j=1,\ldots,n$ and $\ep>0$ we define
$$u_j^\ep(x)=-\frac{d_K(x,F_j)^2}{2}(1+\ep)+R_\O^Kd_K(x,F_j)(1+2\ep).$$
Notice that, if $x\in\O_i\cap\O_j$, then $u_i^\ep(x)=u_j^\ep(x)$.

We denote by $\nu_j$ the outer unit normal to $F_j$. Since $d_K(x,F_j)=\dist(x,F_j) h_K^{-1}(\nu_j)$, the functions $u_j^\ep$ are smooth in the interior of $\O_j$; moreover, since $\nabla\dist(x,F_j)=-\nu_j$, we have
$$\nabla u_j^\ep(x)=\left(\frac{\dist(x,F_j)}{h_K^2(\nu_j)}(1+\ep)-\frac{R_\O^K}{h_K(\nu_j)}(1+2\ep)\right)\nu_j.$$
Thus
\[\begin{split}
\Delta_K(u_j^\ep)&=\dive\left[DH_K(\nu_j) \left(\frac{\dist(x,F_j)}{h_K^2(\nu_j)}(1+\ep)-\frac{R_\O^K}{h_K(\nu_j)}(1+2\ep)\right)\right]\\
&=\frac{1+\ep}{h_K^2(\nu_j)}DH_K(\nu_j)\cdot(-\nu_j)=-1-\ep.
\end{split}\]
Let us consider the function $u^\ep:\O\to\R$ defined by $u^\ep\llcorner_{\O_j} = u_j^\ep\llcorner_{\O_j}$. We now prove that, if $v$ is the solution of the anisotropic torsion problem on $\O$, then $v(x)\le u^\ep(x)$ in $\O$. Suppose, by contradiction, that there exists a point $x\in\O$ such that $u^\ep(x)<v(x)$; since we know that $u^\ep=v$ on $\partial\O$, then the function $u^\ep-v$ has a local minimum, say $x_0$ inside $\O$.

We show that $x_0$ cannot be neither an interior point nor a boundary point of each $\O_j$. If $x_0\in\mathrm{int}\,\O_j$, then $\nabla u^\ep(x_0)=\nabla v(x_0)$, so 
\be\label{gradparal}
D^2H_K(\nabla u^\ep(x_0))=D^2H_K(\nabla v(x_0)).
\ee
Moreover, since $\nabla u^\ep(x_0)\neq 0$ in $\O_j$, then also $\nabla v(x_0)\neq 0$. Thus $v$ solves, in a neighborhood of $x_0$ an elliptic equation with coefficients in $C^{0,\alpha}$, and thence is $C^{2,\alpha}$-regular (see for instance \cite[Chapter 6]{GT}).
Then
\be\label{hesspos}
D^2(u^\ep-v)(x_0)\ge0.
\ee
In particular, using \eqref{gradparal} and \eqref{hesspos}, and since the trace of the product of two positive-definite matrices is positive, we have that
\[\begin{split}
-\ep &=\Delta_Ku^\ep(x_0)-\Delta_Kv(x_0)\\
&=\mathrm{tr}\left[D^2H_K(\nabla u^\ep(x_0))D^2u^\ep(x_0)\right]-\mathrm{tr}\left[D^2H_K(\nabla v(x_0))D^2v(x_0)\right]\\
&=\mathrm{tr}\left[D^2H_K(\nabla u^\ep(x_0))\left(D^2u^\ep(x_0)-D^2v(x_0)\right)\right]\\
&=\mathrm{tr}\left[D^2H_K(\nabla u^\ep(x_0))D^2(u^\ep-v)(x_0)\right]\ge0,
\end{split}\]
which is a contradiction.

We are left to show the contradiction in the case when $x_0\in\O_i\cap\O_j$, $i\ne j$. Let $w(x)=v(x)-v(x_0)+u^\ep(x_0)$; then $w(x_0)=u^\ep(x_0)$ and, thanks to our assumption on $x_0$, there exists a positive number $r$, such that $u^\ep(x)\ge w(x)$, for every $x\in(x_0+rK)\subset\O$.

Since, $u_j^\ep$ is a concave function of the distance, we have that
\be\label{superdiff}
u_j^\ep(x)\le u_j^\ep(x_0)+\nu_j\cdot(x-x_0)\left(\frac{\dist(x,F_j)}{h_K^2(\nu_j)}(1+\ep)-\frac{R_\O^K}{h_K(\nu_j)}(1+2\ep)\right)
\ee
for every $x\in(x_0+rK)$. Without loss of generality we can choose $r<R_\O^K\ep/(1+\ep)$; this choice of $r$ allows us to conclude that
$$d_K(x,F_j)\le r+d_K(x_0,F_j),$$
and thus the factor
$$\frac{\dist(x,F_j)}{h_K^2(\nu_j)}(1+\ep)-\frac{R_\O^K}{h_K(\nu_j)}(1+2\ep),$$
appearing in the right-hand side of \eqref{superdiff}, never vanishes, for every $x\in(x_0+rK)$. Indeed
\[\begin{split}
\frac{\dist(x,F_j)}{h_K^2(\nu_j)}(1+\ep)-\frac{R_\O^K}{h_K(\nu_j)}(1+2\ep)
&\le\frac{r(1+\ep)}{h_K(\nu_j)}+\frac{d_K(x_0,F_j)(1+\ep)}{h_K(\nu_j)}-\frac{R_\O^K}{h_K(\nu_j)}(1+2\ep)\\
&<\frac{d_K(x_0,F_j)(1+\ep)}{h_K(\nu_j)}-\frac{R_\O^K}{h_K(\nu_j)}(1+\ep)\le 0.
\end{split}\]
Moreover, since $r<R_\O^K\ep/(1+\ep)$, for any other $l$ such that $x_0\in\O_l$, inside $x_0+rK$, $u_l^\ep$ is an increasing function of the anisotropic distance $d_K$, then it follows from the definition of the sets $\O_l$ that
$$u_l^\ep(x)\le u_j^\ep(x)\qquad\hbox{for every }x\in (x_0+rK)\cap\O_l,$$
that entails that
\be\label{u-est}
u^\ep(x)\le u_j^\ep(x)\qquad\hbox{for every }x\in(x_0+rK).
\ee

Plugging \eqref{u-est} into \eqref{superdiff} and recalling that $u^\ep(x_0)=u_j^\ep(x_0)$, we obtain
\be\label{supdiffj}
u^\ep(x)\le u^\ep(x_0)+\nu_j\cdot(x-x_0)\left(\frac{\dist(x,F_j)}{h_K^2(\nu_j)}(1+\ep)-\frac{R_\O^K}{h_K(\nu_j)}(1+2\ep)\right)
\ee
for every $x\in(x_0+rK)$.

Arguing analogously for the function $u_i^\ep$ we find
\be\label{supdiffi}
u^\ep(x)\le u^\ep(x_0)+\nu_i\cdot(x-x_0)\left(\frac{\dist(x,F_i)}{h_K^2(\nu_i)}(1+\ep)-\frac{R_\O^K}{h_K(\nu_i)}(1+2\ep)\right)
\ee
for every $x\in(x_0+rK)$. By imposing in \eqref{supdiffj} and \eqref{supdiffi} the conditions $v\le u^\ep$ and $v=u^\ep$ at $x_0$, we find
$$v(x)- v(x_0)\le\nu_j\cdot(x-x_0)\left(\frac{\dist(x,F_j)}{h_K^2(\nu_j)}(1+\ep)-\frac{R_\O^K}{h_K(\nu_j)}(1+2\ep)\right)$$
and
$$v(x)- v(x_0)\le \nu_i\cdot(x-x_0)\left(\frac{\dist(x,F_i)}{h_K^2(\nu_i)}(1+\ep)-\frac{R_\O^K}{h_K(\nu_i)}(1+2\ep)\right).$$
Since $v$ is differentiable at $x_0$, by multiplying both sides by the factor $|x-x_0|^{-1}$ and taking the limit as $x\to x_0$ we obtain that $\nabla u(x_0)$ must be proportional both to $\nu_i$ and $\nu_j$, that is a contradiction.

We can now use the function $u_j$ as a barrier function to get an estimate of the torsion rigidity of $\O$. In order to compute $\int_{\O_j}u_j^\ep(x)\,dx,$ it is important to make sure that each subdomain $\O_j$ can be written as a graph over the facet $F_j$. This is always the case when $K$ is the Euclidean ball, since its radius is orthogonal to the tangent space, while this is not true for a general convex body $K$.

To get the same condition, we consider a linear transformation $L$ (see Figure \ref{fig2}), such that $L\llcorner\nu_j^\perp=\mathrm{Id}$ and $L Dh_K(\nu_j)=h_K(\nu_j)\nu_j$. Notice that, since $Dh_K(\nu_j)\cdot\nu_j=h_K(\nu_j)>0$, $L$ is well defined and
\be\label{determinante}
\mathrm{det}\,L=1.
\ee
Moreover $h_K(\nu_j)=h_{LK}(\nu_j)$, indeed
\be\label{supplk}
h_K(\nu_j)=\sup_{x\in K}x\cdot\nu_j=\sup_{y\in LK}L^{-1}y\cdot\nu_j=\sup_{y\in LK}\left(L^{-1}y'\cdot\nu_j+(y\cdot\nu_j)L^{-1}\nu_j\cdot\nu_j\right),
\ee
where $y=y'+(y\cdot\nu_j)\nu_j$, and $y'\cdot \nu_j=0$. Since $L\llcorner\nu_j^\perp=\mathrm{Id}$, then $L^{-1}y'\cdot\nu_j$, and since $Dh_K(\nu_j)=h_K(\nu_j)L^{-1}\nu_j$, then $L^{-1}\nu_j\cdot\nu_j=1$ so that, by recalling \eqref{supplk}
\[
h_K(\nu_j)=\sup_{y\in LK}y\cdot\nu_j=h_{LK}(\nu_j).
\]
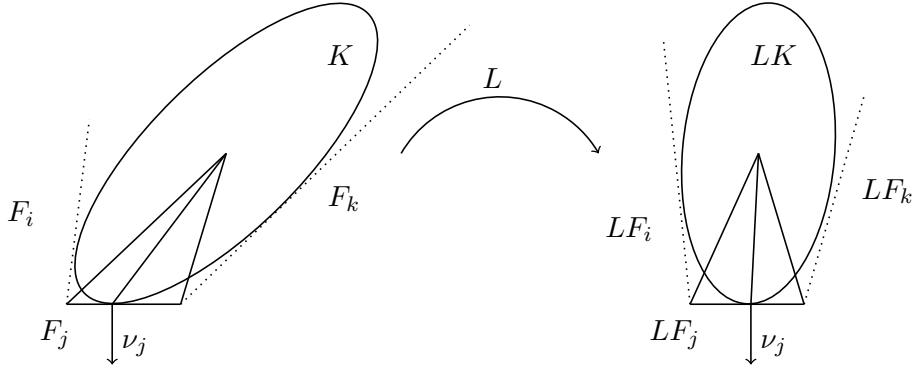
\begin{figure}[h]
\begin{tikzpicture} [semithick]

\draw [rotate around={-45:(0,2)}] (0,2) ellipse (1 and 2.63);
\draw (-2.1,0) -- (-0.6,0);
\draw (0,2) -- (-2.1,0);
\draw (0,2) -- (-0.6,0);
\draw [dotted] (-0.6,0) -- (3.2,3.7);
\draw [dotted] (-2.1,0) -- (-1.8, 2.45);
\node at (-2.25,-0.4) {$F_j$};
\node at (-2.7,1.2) {$F_i$};
\node at (1.55,1.4) {$F_k$};
\draw (0,2) -- (-1.5,0);

\draw[->] (-1.5,0) -- (-1.5,-0.8);
\node at (-1.2,-0.55) {$\nu_j$};

\draw [->] (2.3,2) arc (150:30:1.5);
\node at (3.5,3) {$L$};
\node at (1.5,3.3) {$K$};
\node at (7.2,3.3) {$LK$};

\draw (7,2) -- (6.1,0);
\draw (7,2) -- (7.6,0);
\draw [rotate around={-5:(7,2)}] (7,2) ellipse (1 and 2);
\draw (6.1,0) -- (7.6,0);
\draw [ dotted]  (6.1,0) -- (5.75,3.5);
\draw [dotted]  (7.6,0) -- (8.4, 2.8);
\node at (5.9,-0.4) {$LF_j$};
\node at (5.3,1) {$LF_i$};
\node at (8.7,1.5) {$LF_k$};
\draw (7,2) -- (6.9,0);
\draw[->] (6.9,0) -- (6.9,-0.8);
\node at (7.2,-0.55) {$\nu_j$};
\end{tikzpicture}

\caption{The linear transformation $L$} \label{fig2}
\end{figure}

We show that our choice of the linear application $L$ allows us to describe the subdomain $L\O_j$ as
$$L\O_j=\big\{y+r\nu_j\ :\ y\in LF_j,\ r\in[0,\rho_j(y)]\big\}.$$
For this we need only to prove that $L\O_j$ is contained in the cylinder $C_j=LF_j\times \R\nu_j$, since $L\O_j$ is a convex set. We start noticing that, since $\dist_K(x,F_j)=\dist_{LK}(Lx,LF_j)$,
$$L\O_j=\big\{x\in L\O\ :\ d_{LK}(x,LF_j)\le d_{LK}(x,LF_l),\mbox{ for every }l\ne j\big\}.$$

Suppose, by contradiction, that there exists a point $\overline x\in L\O_j\setminus C_j$. Let $\overline r$ be the largest number such that $\overline x +\overline rLK\subseteq L\O$. It follows from the definition of $L\O_j$, that there exists a point $\overline y\in\overline x +\overline rLK\cap LF_j$.

Since $\overline x +\overline rLK\subseteq L\O$, the normal cone of $L\O$ is contained in the normal cone of $\overline x + \overline rLK$ at the point $\overline y$, and then $\nu_j$ is the outer normal unit to $\overline x +\overline rLK$ at $\overline y$. Therefore $\overline x=\overline y+\overline r\nu_j\in C_j$.

Now we compute
\[\begin{split}
\int_{\O_j}u_j^\ep(x)\,dx
&=\int_{L\O_j}u_j^\ep(L^{-1}y)\,dy\\
&=\int_{L\O_j}\left(-\frac{\dist^2(x,F_j)}{2h_{K}^2(\nu_j)}(1+\ep)+\frac{R_\O^K\dist(x,F_j)}{h_{K}(\nu_j)}(1+2\ep)\right)dx.
\end{split}\]
We used \eqref{determinante} for the first equality, we use the fact that $\dist_K(x,F_j)=\dist_{LK}(Lx,LF_j)$ and $h_K(\nu_j)=h_{LK}(\nu_j)$ for the second one. Finally
\[\begin{split}
\int_{\O_j}u_j^\ep(x)\,dx
&=\int_{L\O_j}\left(-\frac{\dist^2(x,F_j)}{2h_{K}^2(\nu_j)}+\frac{R_\O^K\dist(x,F_j)}{h_{K}(\nu_j)}\right)dx+O(\ep)\\
&=\int_{LF_j}\int_0^{\rho_j(y)}\left(-\frac{\dist^2(y+r\nu_j,F_j)}{2h_K^2(\nu_j)}+\frac{R_\O^K\dist(y+r\nu_j,F_j)}{h_K(\nu_j)}\right)dr\,dy + O(\ep)\\
&=\int_{LF_j}\int_0^{\rho_j(y)}\left(-\frac{r^2}{2h_K^2(\nu_j)}+\frac{R_\O^Kr}{h_K(\nu_j)}\right)dr\,dy + O(\ep)\\
&=\int_{LF_i}\left(-\frac{\rho_j^3(y)}{6h_K^2(\nu_j)}+\frac{R_\O^K\rho_j^2(y)}{2h_K(\nu_j)}\right)dy+O(\ep).
\end{split}\]
Since $\rho_j(y)\le R_\O^Kh_K(\nu_j)$, for every $y\in LF_j$, then
\be\label{2degeq}
-\frac{\rho_j^2}{6h_K^2(\nu_j)}+R_\O^K\frac{\rho_j(y)}{2h_K(\nu_j)} \le \frac13\left(R_\O^K\right)^2.
\ee
Using \eqref{2degeq} we find
\[\begin{split}
\int_{\O_j}u_j^\ep(x)\,dx
&\le\int_{LF_j}\frac{\rho_j(y)\left(R_\O^K\right)^2}{3}\,dy+O(\ep)\\
&=\frac13\left(R_\O^K\right)^2 |L\O_j|+O(\ep).
\end{split}\]
Since $|L\O_j|=|\O_j|$, we get
\[\begin{split}
T^K(\O)&\le\int_\O u^\ep\,dx=\sum_j\int_{\O_j}u_j^\ep\,dx\\
&\le\frac13\left(R_\O^K\right)^2\sum_j|\O_j|+O(\ep)=\frac13\left(R_\O^K\right)^2|\O|+O(\ep)
\end{split}\]
as required.
\end{proof}

{\bf Remark.} The most important inequality playing a role in the proof of the theorem above is \eqref{2degeq}. This inequality must be strict for some boundary point, unfortunately this is not enough to conclude that equality in \eqref{upboundtors} is never attained, because our proof is based on an approximation procedure.
In order to get the strict sign in \eqref{upboundtors} we must to take into account and estimates the reminder terms in \eqref{2degeq}.

\[
-\frac{\rho_j^2}{6h_K^2(\nu_j)}+R_\Omega^K\frac{\rho_j(y)}{2h_K(\nu_j)} = \frac13\left(R_\Omega^K\right)^2 -\frac1 6\left(R_\Omega^K-\frac{\rho_j}{h_K(\nu_j)}\right)^2+\frac{\rho_j R_\Omega^K}{6h_K(\nu_j)}-\frac{(R_\Omega^K)^2}{6}
\]
\[
\le \frac13\left(R_\Omega^K\right)^2+\frac{\rho_j R_\Omega^K}{6h_K(\nu_j)}-\frac{(R_\Omega^K)^2}{6}.
\]

Arguing as in the proof above we can conclude that

\[
\int_{\Omega_j}v\le \frac13\left(R_\Omega^K\right)^2|\Omega|+\frac 1 6 \int_{F_j}\left(\frac{\rho_j^2 R_\Omega^K}{h_K(\nu_j)}-\rho_j(R_\Omega^K)^2\right)dy
\]
\[
= \frac13\left(R_\Omega^K\right)^2|\Omega_j|+\frac1 6\int_{\Omega_j}\left(2R_\Omega^K\mathrm{dist}_K(x,F_j)-(R_\Omega^K)^2\right)dx.
\]
Summing over $j$ we find
\begin{equation}\label{nonexistence}
T^K(\Omega)\le  \frac13\left(R_\Omega^K\right)^2|\Omega_j|+\frac{R_\Omega^K}{6}\int_\Omega \left(d_K(x,\partial \Omega)-R_\Omega^K\right) dx.
\end{equation}

Taking into account Equation \eqref{nonexistence} in order to show that the minimum in \eqref{upboundtors} is never attained we have just to show that, for every convex body $\Omega$
\[
2\int_\Omega d_K(x,\partial \Omega)< R_\Omega^K|\Omega|.
\]

Let us denote by $\Omega_t=\{x\in\Omega\,:\, d_K(x,\partial\Omega)\ge t\}$. Since $|\nabla d_K(x)|=h_K(\nu_t(x))^{-1}$, where $\nu_t(x)$ is the outer normal to $\partial\Omega_t$ at $x$,
thanks the co-area formula we can write
\[
\int_\Omega d_K(x,\partial \Omega)=\int_0^{R_\Omega^K}\int_{\partial\Omega_t}d_k(x)h_K(\nu_t(x))\,d\mathcal H^{N-1}(x) dt=\int_0^{R_\Omega^K}tP_K(\Omega_t)\,dt,
\]
where $P_K$ denotes the anisotropic perimeter.\\
On the other hand 
\[
\int_0^{R_\Omega^K}d_K(x,\partial \Omega)\,dx=\int_0^{R_\Omega^K}|\Omega_t|\,dt.
\]
As it is well known $tP_K(\Omega_t)+|\Omega_t|<|\Omega_t+tK|$; this fact can be seen as a consequence of the characterization of mixed volumes (see, for instance \cite[Theorem 5.1.7]{Scn}).
Since $\Omega_t+tK\subseteq \Omega$, we have
\[
2\int_\Omega d_K(x,\partial \Omega)<\int_0^{R_\Omega^K}|\Omega_t+tK|\,dt\le\int_0^{R_\Omega^K}|\Omega|\,dt=R_\Omega^K|\Omega|.
\]

Nonetheless, the following example tells us that \eqref{upboundtors} is sharp.

\begin{prop}\label{rect}
Let $\O^\varepsilon$ the rectangle $[-\varepsilon,\ep]\times[-a_2,a_2]\times\ldots\times[-a_N,a_N]$. Then
$$\lim_{\ep\to 0^+}\frac{T^K(\O^\ep)}{\left(R^K_{\O^\ep}\right)^2|\O^\ep|}=\frac13\;.$$
\end{prop}

\begin{proof}
Let $\O^\ep=C^\ep\cup D^\ep$, where $C^\ep=[-\varepsilon,\ep]\times[-a_2+\ep,a_2-\ep]\times\ldots\times[-a_N+\ep,a_N-\ep]$, and $D^\ep= {\O^\ep}\setminus C^\ep$, as in Figure \ref{fig1}. Setting $x=(x_1,z)$ with $z\in\R^{N-1}$ and $a=(a_2,\ldots,a_N)$, we consider the function $u^\ep$ defined by
$$\begin{cases}
u^\ep(x_1,z)=\ds\frac{\ep^2-x_1^2}{2h^2_K(e_1)}&\mbox{in }C^\ep\\
u^\ep(x_1,z)=\ds\min\big\{|a-z|,|-a-z|\big\}\frac{\ep^2-x_1^2}{2\ep h^2_K(e_1)}&\mbox{in }D^\ep.
\end{cases}$$

\begin{figure}[h]
\begin{center}
\begin{tikzpicture}
\draw (0,0) rectangle (1,1);
\draw (1,0) rectangle (4,1);
\draw (4,0) rectangle (5,1);
\draw [loosely dotted, gray] (-1,0.5) -- (6,0.5);
\draw [loosely dotted, gray](2.5,-0.5) -- (2.5,2);
\node at (2.8,-0.2) {-$\varepsilon$};
\node at (-0.2,-0.2) {-$a_2$};
\node at (1,-0.2) {-$a_2+\varepsilon$};
\node at (4,-0.2) {$a_2-\varepsilon$};
\node at (2.8,1.2) {$\varepsilon$};
\node at (5.2,-0.2) {$a_2$};
\node at (6,0.2) {$\R$};
\node at (2.8, 2.2) {$\R^{N-1}$};
\node at (0.5,0.5) {$D^\eps$};
\node at (4.5,0.5) {$D^\eps$};
\draw [bend left] (3.1,0.6) to node {} (4.6,2.5);
\node at (4.9, 2.5) {$C^{\varepsilon}$};
\end{tikzpicture}
\end{center}
\caption{The set $\O^\eps=C^\eps \cup D^\eps$} \label{fig1}
\end{figure}

We can estimate the $K$-torsional rigidity
$$T^K(\O^\ep)\ge\frac{\left(\int_{\O^\ep}u^\ep\right)^2}{\int_{\O^\ep} h^2_K(\nabla u^\ep)}=\frac{\left(\int_{C^\ep}u^\ep+\int_{D^\ep}u^\ep\right)^2}{\int_{C^\ep}h^2_K(\nabla u^\ep)+\int_{D^\ep}h^2_K(\nabla u^\ep)}\;.$$
We now compute
$$\int_{C^\ep}u^\ep\,dx=\int_{C^\ep}\frac{\ep^2-x_1^2}{2h^2_K(e_1)}\,dx=\frac{|C^\ep|\ep^2}{3h^2_K(e_1)}$$
and
$$\int_{C^\ep}h^2_K(\nabla u^\ep)\,dx=\int_{C^\ep}\frac{x_1^2}{h^2_K(e_1)}\,dx=\frac{|C^\ep|\ep^2}{3h^2_K(e_1)}\;.$$
We notice that both $\int_{D^\ep}u^\ep\,dx$ and $\int_{D^\ep}H_K(\nabla u^\ep)\,dx$ are negligible, since they go to zero as $\ep^{N+2}$; moreover $|D^\ep|$ is negligible as well, since it goes to zero as $\ep^N$. By recalling that
$$\left(R^K_{\O^\ep}\right)^2=\frac{\ep^2}{h^2_K(e_1)}\;,$$
we have that
$$\lim_{\ep\to0}\frac{T^K(\O^\ep)}{\left(R^K_{\O^\ep}\right)^2|\O^\ep|}
\ge
\lim_{\ep\to0}\frac{\left(\int_{C^\ep}u^\ep\,dx\right)^2}{\left(R^K_{\O^\ep}\right)^2|\O^\ep|\int_{C^\ep}h^2_K(\nabla u^\ep)\,dx}=\frac13$$
which, together with Theorem \ref{upt}, concludes the proof.
\end{proof}

\section{Eigenvalue}\label{eigenvalue}

In this section we extend to the anisotropic case inequalities \eqref{easyest} and \eqref{hersch}. As in the linear case, the upper bound for the principal frequency relies on the monotonicity and scaling laws. As we mentioned, all the best constant in those inequalities do not depend on $K$; in the following proposition we show that $\lambda_1^K(K)=j_0^2$, and this gives the second inequality in \eqref{eigen}.

As is Section \ref{torsione}, we always assume that $K$ is a $C^2_+$ regular symmetric convex body.

\begin{prop}\label{lambdaball}
Let $\O\subset \R^N$ be a convex body, then
$$\lambda_1^K(\O)\le \lambda_1(B) (R_\O^K)^{-2}.$$
\end{prop}

\begin{proof}
From \eqref{monotonicity} and \eqref{scaling} we have that $$\lambda_1^K(\O)\le\lambda_1^K(K) (R_\O^K)^{-2}.$$
Now have to show that $\lambda_1^K(K)\lambda_1(B)$. Since there exists a positive function $g$ solving
$$-\Delta_K g=\lambda_1^K g\mbox{ in }K,\qquad g=0\mbox{ on }\partial K,$$
we can assume that $g$ is radial, namely $g(x)=f\big(H^*(x)\big)$ (see \cite{aflt}). A direct computation shows that
$$f''(t)+Nf'(t)=-\lambda_1^K f(t)\qquad\hbox{and}\qquad f(1)=0.$$
Indeed
\[\begin{split}
-\lambda_1^K g
&=\Delta_K g=\dive\big(DH_K(\nabla g)\big)=\dive\big(DH_K\big(f'(H_K^*)DH_K^*\big)\big)\\
&=\dive\big(f'(H_K^*)x\big)=f''(H_K^*)DH_K^*(x)\cdot x + N f'(H_K^*).
\end{split}\]
Setting $\overline g(x)= f(|x|^2/2)$ it easy to check that
$$\Delta\overline g=-\lambda_1^K\overline g\mbox{ in }B,\qquad\overline g=0\mbox{ on }\partial B.$$
Since $\overline g>0$, then $g$ is the first Dirichlet eigenfunction of the ball $B$ and thence $\lambda_1^K(K)=\lambda_1(B)$.
\end{proof}

Now we prove three technical lemmas that are useful in the proof of the lower bound for the principal frequency. Lemmas \ref{gale1} and \ref{gale2} are generalizations of a lemma by D. Gale and referred to as a private communication by M.H. Protter in \cite{pro}.

\begin{lemma}\label{gale1}
Let $\O$ be a convex domain, let $K\subseteq\Omega$ be a convex body with differentiable boundary, and suppose that for every $\tau>1$ and $x\in\R^N$, $x+\tau K\not\subseteq\O$. Then there exist $i\in\N$ and $x_1,\ldots,x_i\,\in\partial\O\cap\partial K$ such that, denoted by $\nu_j$ the outer unit normal to $\partial K$ at $x_j$, for every $\nu\in\mathbb S^{N-1}$, there exists, $j$, such that $\nu\cdot\nu_j\ge 0$.
\end{lemma}

\begin{proof}
Let $\nu\in\mathbb S^{N-1}$; for $\ep>0$, we consider the set $K^\ep=K+\ep\nu$. Notice that $K^\ep$ cannot be a proper subset of $\Omega$, since, otherwise, also a small dilation of $K^\ep$ should be, against our assumptions. Then, there exists $\mu^\ep\in\mathbb S^{N-1}$ such that
$$h_\O(\mu^\ep)\le h_{K^\ep}(\mu^\ep)=h_K(\mu^\ep)+\ep\mu^\ep\cdot\nu.$$ Notice that $\mu^\ep$ is the outer unit normal to $\partial K^\ep$ at some point $x^\ep\in K^\ep\setminus \O$. Since $K\subseteq\O$, then $h_K\le h_\O$, and then $\mu^\ep\cdot\nu\ge 0$.

Up to extract a subsequence we can always assume that, as $\ep$ goes to $0$, $\mu^\ep$ converges to some unit vector $\mu$ and $x^\ep$ converges to some point $x\in\partial K\cap \partial \O$ with the property that $\mu\cdot \nu\ge 0$.

We have proved that $C_x=\big\{\nu\in\mathbb S^{N-1}\ :\ \nu\cdot\nu_x\ge0\big\}$, $x\in\partial K\cap\partial\O$ is a covering of the unit sphere; we are left to show that we can extract from $C_x$ a finite covering of the sphere.

We proceed by induction on the dimension $N$. If $N=1$ there is nothing to prove, since the sphere $\mathbb S^0$ is a finite set. Let now $N>1$ and let $O_x=\big\{\nu\in\mathbb S^{N-1}\ :\ \nu\cdot\nu_x>0\big\}$. If $\mathbb S^{N-1}\subseteq\cup_{x}O_x$, then, by compactness, we can extract a finite covering of the sphere $\mathbb S^{N-1}$. Otherwise, there exists a unit vector $\nu$, such that $\nu\cdot\nu_x\le 0$, for every $x\in\partial K\cap \partial \O$.

Let $A=\big\{x\in\partial K\cap\partial\O\ :\ \nu_x\cdot\nu=0\big\}$; we claim that $\{C_x\}_{x\in A}$ is a covering of $\mathbb S^{N-1}\cap\nu^\perp=\mathbb S^{N-2}$. Let $\eta\in\mathbb S^{N-1}$, such that $\eta\cdot\nu=0$ and let, for $\ep>0$, $\eta^\ep=\sin(\ep)\eta+\cos(\ep)\nu$. Let $x^\ep\in\partial K\cap\partial\O$, such that $\nu_{x^\ep}\cdot\eta^\ep\ge0$. Since $\nu_{x^\ep}\cdot\nu\le 0$, we must have $\nu_{x^\ep}\cdot\eta\ge0$.

Up to extracting a subsequence $x^\ep$ converges to some point $x$. Clearly $\nu_x\cdot\nu\ge0$ and, since $\partial K\cap\partial\O$ is closed, also $\nu_x\cdot\nu\le0$ and thus $x\in A$. The fact that $\nu_x\cdot\eta\ge0$ proves our claim.

We can finally apply our induction hypothesis to find a finite subset $A'\subseteq A$ such that $\mathbb S^{N-2}\subset\cap_{x\in A'}C_x$, and it is straightforward to check that also $\mathbb S^{N-2}\subset\bigcap_{x\in A'}C_x$.
\end{proof}

\begin{lemma}\label{gale2}
Let $\O$ and $K$ as above, then there exists a polyhedral convex domain $T$ such that $\O\subseteq T$ and, for every $\tau>1$ and $x\in\R^N$, $x+\tau K\not\subseteq\O$. 
\end{lemma}

\begin{proof}
Let $x_1,\ldots,x_i$ and $\nu_1,\ldots,\nu_i$ be as in Lemma \ref{gale1}.
We set $T=\bigcap \{x\,:\,(x-x_i)\cdot \nu_i\le0\}$.
Since $T$ is the intersection of supporting half-spaces of $\O$, then $\O\subseteq T$.

Suppose now, that there exists $\overline x$, $\tau>1$ such that $\overline x+\tau K\subseteq T$. Let $j$ be such that $\nu_j\cdot\overline x\ge0$. The point $x'_j=\tau x_j+\overline x\in\overline x+\tau K$ verifies
$$(x'_j-x_j)\cdot\nu_j\ge (\tau-1)x_j\cdot \nu_j>0,$$
hence it cannot belong to $T$.
\end{proof}

\begin{lemma}\label{linear}
For every invertible linear mapping $L:\R^N\to\R^N$ we have
$$\lambda_1^K(\O) = \lambda_1^{LK}(L\Omega).$$
\end{lemma}

\begin{proof}
By the definition of support function, we have
$$h_{L^{-1}K}(x)=\sup_{y\in L^{-1}K}\langle y,x\rangle=\sup_{z\in K}\langle L^{-1}z,x\rangle=\sup_{z\in K}\langle z,L^{-t}x\rangle=h_K(L^{-t}x),$$
where $L^{-t}$ denotes the inverse of the transpose of $L$.\\
Denoting by $v(x)=u(L^{-1}x)$, again by a simple computation we have
$$\nabla v(x)=L^{-t}\nabla u(L^{-1}x)$$
and
$$h_{LK}(\nabla v(x))=h_K(L^t\nabla v(x))=h_K(\nabla u(L^{-1}x)).$$
Then, using the computation above and the change of variable formula, we have
$$\frac{\int_{L\O}h^2_{LK}(\nabla v)\,dx}{\int_{L\O}v^2\,dx}
=\frac{\int_{L\O}h^2_K(\nabla u\circ L^{-1})\,dx}{\int_{L\O}(u\circ L^{-1})^2\,dx}=\frac{\int_\O h^2_K(\nabla u)\,dx}{\int_\O u^2\,dx},$$
proving the lemma.
\end{proof}

We are now ready to prove the validity of \eqref{eigen}. The proof is carried out in the same spirit of the proof of the result achieved by Hersch in \cite{her}, even if the introduction of the anisotropy obliges us to avoid to use Hersch's estimates and to adapt the proof.

\begin{teo}
Let $\O\subset\R^N$ be a convex body, then
\[
\lambda_1^K(\O)\ge \frac{\pi^2}{4}({R_\O^K})^{-2}.
\]
\end{teo}

\begin{proof}

Up to a translation and a dilation of the domain, we can assume that $\O$ satisfies the assumptions of Lemma \ref{gale1} (notice that in this case, $R_\O^K=1$) by repeating the construction explained in Lemmas \ref{gale1} and \ref{gale2} we find a convex polyhedral domain $T\supseteq\O$ with the same anisotropic unitary inradius. For $j=1,\ldots,i$ we call $T_j$ the pyramidal domain obtained by taking the convex hull of the origin and the $j$-th facet $F_j$ of $T$, and we set $\O_j=\O\cap T_j$.

Let now $u$ be the first eigenfunction, we compute
\[\begin{split}
\lambda_1^K(\O)&=\frac{\int_\O h^2_K(\nabla u)\,dx}{\int_\O u^2\,dx}=\frac{\sum_j\int_{\O_j}h^2_K(\nabla u)\,dx}{\sum_j\int_{\O_j}u^2\,dx}\\
&\ge\min_j\frac{\int_{\O_j}h^2_K(\nabla u)\,dx}{\int_{\O_j}u^2\,dx}=\min_j\frac{\int_{T_j}h^2_K(\nabla u)\,dx}{\int_{T_j}u^2\,dx}\;.
\end{split}\]
As we shall see, to conclude it is enough to get an estimate from below of the quantity
$$\frac{\int_{T_j}h^2_K(\nabla u)\,dx}{\int_{T_j}u^2\,dx}\;.$$
To this aim, as we did in Theorem \ref{upboundtors}, we consider a linear map $L$ such that $L\llcorner \nu_j^\perp=\mathrm{Id}$ and $L Dh_K(\nu_j)=h_K(\nu_j)\nu_j$. Arguing as in Lemma \ref{linear}, it is easy to show that
$$\frac{\int_{T_j}h^2_K(\nabla u)\,dx}{\int_{T_j}u^2\,dx}=\frac{\int_{LT_j}h^2_{LK}(\nabla f)\,dx}{\int_{LT_j}f^2\,dx}\;,$$
where $f(x)=u(L^{-1}x)$. Moreover it is clear that $LT_j$ is a graph over $F_j$.

For any function $f$ defined on $T_j$ we have
\be\label{stima1}
h^2_{LK}(\nabla f)\ge h^2_{LK}(\nu_j) (\nabla_{\nu_j}f)^2.
\ee
Indeed, since $Dh_K(\nu_j)\in K$, then $h_K(\nu_j)\nu_j\in LK$ and thence, by recalling that $h_K(\nu_j)=h_{LK}(\nu_j),$ we have
$$h_{LK}(\nabla f)\ge h_K(\nu_j)\nu_j\cdot\nabla f=h_{LK}(\nu_j)\nabla_{\nu_j}f.$$
We recall that, for any test function $v(t):[0,\ell]\to\R$, such that $v(0)=0$ it holds true that
\be\label{eig1d}
\frac{\int_0^\ell v'(t)^2\,dt}{\int_0^\ell v(t)^2\,dt}\ge\frac{\pi^2}{4\ell^2}\;.
\ee
Thanks to \eqref{stima1} and \eqref{eig1d}, we can conclude by computing
\[\begin{split}
\frac{\int_{T_j}h^2_K(\nabla u)\,dx}{\int_{T_j}u^2\,dx}
&=\frac{\int_{LF_j}\int_0^{\ell(y)}h^2_K(\nabla f)\,dy\,dt}{\int_{LF_j}\int_0^{\ell(y)}f^2\,dy\,dt}\ge\frac{\int_{LF_j}\int_0^{\ell(y)}h^2_K(\nu_j)\nabla_{\nu_j}f\,dy\,dt}{\int_{LF_j}\int_0^{\ell(y)}f^2\,dy\,dt}\\
&=h^2_K(\nu_j)\left[\int_{LF_j}\frac{\int_0^{\ell(y)}\nabla_{\nu_j}f\,dt}{\int_0^{\ell(y)}f^2\,dt}\frac{\int_0^{\ell(y)}f^2\,dt}{1}\,dy\right]\frac{1}{\int_{LF_j}\int_0^{\ell(y)}f^2\,dy\,dt}\\
&\ge h^2_K(\nu_j)\left[\int_{LF_j}\frac{\pi^2\int_{0}^{\ell(y)}u^2\,dt}{4\ell(y)^2}\right]\frac{1}{\int_{LF_j}\int_0^{\ell(y)}u^2\,dy\,dt}\ge h^2_K(\nu_j)\frac{\pi^2}{4\ell_{\max}^2},
\end{split}\]
where $\ell_{\max}$, by construction, is $R_\O^Kh_K(\nu_j)$.
\end{proof}

\section{Regularity}\label{reg}

In the previous sections we limit ourselves to the case of norms associated to a $C_+^2$-regular convex body $K$. In Section \ref{prel}, we mentioned that this assumption is necessary to make the Finsler Laplace operator uniformly elliptic, nonetheless the definitions of the $K$-torsional rigidity and the $K$-principal frequency via \eqref{torsion} and \eqref{lambda1} make sense even for less regular norms $h_K$. Moreover, several interesting examples of applications for those inequalities in the case of non-Euclidean norms arise when we consider the norms associated to a square, or more in general, $p$-norm of the form
$$\|v\|_p=\left(\sum|v_i|^p\right)^{1/p}.$$

Notice that, for $1\le p<2$ and $p=\infty$, the function $\|\cdot\|_p^2$ is not $C^2$-regular, while for $2<p<\infty$ it is smooth but its Hessian may vanish. Thus all these cases are not covered by the theorems that we proved so far.

In this section we remove any regularity assumption on the convex body $K$. 
This result, contained in Theorem \ref{regularity}, relies on a result by the first author and Dal Maso (see \cite{bdm}) concerning the $\Gamma$-convergence (we refer to \cite{dm93} for an exhaustive review on the topic) of some integral functionals of the form
$$F_h(u)=\int_\O f_h(x,u,Du)\,dx\;,$$
to a functional
$$F(u)=\int_\O f(x,u,Du)\,dx$$
by means of the pointwise convergence of the integrand functions $f_h$ to a function $f$.

For the reader convenience let us state the above mentioned result.

\begin{teo}[\cite{bdm}]\label{thbdm}
Assume that all $f_h=f_h(x,u,p)$ are convex in $p$ and that the sequence $(f)_h$ converges pointwise to a function $f$. Assume that there exist two increasing continuous functions $\mu,\nu:\R^+\to\R^+$, with $\mu(0)=\nu(0)=0$ and $\mu$ concave, such that
$$|f_h(x,u,p)-f_h(y,v,p)|\le\nu(|x-y|)(1+f_h(x,u,p))+\mu(|u-v|),$$
for each $x,y\in\O$, $u,v\in\R$, $p\in\R^n$. Then
$$\Gamma\hbox{-}\lim F_h(u) = F(u).$$
\end{teo}

Our strategy is the following: we approximate a convex body $K$ with a sequence of $C_+^2$-regular convex bodies $K_n$, and we pass to the $\Gamma$-limit the desired inequalities.

\begin{teo}\label{regularity}
Let $K$ and $\O$ be convex bodies, then \eqref{torsion1} and \eqref{eigen} hold true.
\end{teo}

\begin{proof}
We limit ourselves to prove \eqref{eigen}, being the proof of \eqref{torsion1} completely analogous. Let $K_n$ be a sequence of $C_+^2$-regular convex bodies converging to $K$ with respect to the Hausdorff metric (the existence of such a sequence is an old theorem by Minkowski, but we refer to \cite{Sc} for a shorter proof), and let 
$$F_n(u)=\int_\O h^2_{K_n}(\nabla u)\,dx\;,$$
$u\in X$, where $X$ is the functional space made of all functions in $W_0^{1,2}(\O)$, with $L^2$ norm equal to 1. The minimizers of the functionals $F_n$, say $u_n$, satisfy
\be\label{approx}
\frac{\pi^2}{4}(R_\O^{K_n})^{-2}\le F_n(u_n)\le j_0^2 (R_\O^{K_n})^{-2}.
\ee
Notice that $R_\O^{K_n}\to R_\O^K$. Moreover, since $h_{K_n}\to h_K$ and since $h_K(v)\le c\|v\|$, for some positive constant $c$, then the second inequality in \eqref{approx} gives the equi-boundedness of the minimizers $u_n$ in $W^{1,2}$, that ensures the existence, up to extracting a subsequence, of a weak limit, $\overline u$.

Since the functionals $F_n$ satisfy the assumptions of Theorem \ref{thbdm}, we have that $F(u)=\int_\O h^2_K(\nabla u)\,dx$ is the $\Gamma$-limit of the functionals $F_n$ and, since minimizers converge to minimizers, we obtain
$$\lambda_1^K(\O)=\min F(u)=F(\overline u)=\lim F_n(u_n).$$
Since both the right-hand side and the left-hand side of \eqref{approx} are converging, it is possible to pass to the limit those inequalities to find the desired result.
\end{proof}

\ack
The work of the first and second authors is part of the project 2015PA5MP7 {\it``Calcolo delle Variazioni''} funded by the Italian Ministry of Research and University. The third author is supported by the MIUR SIR-grant {\it``Geometric
Variational Problems''} (RBSI14RVEZ). The authors are members of the Gruppo Nazionale per l'Analisi Matematica, la Probabilit\`a e le loro Applicazioni (GNAMPA) of the Istituto Nazionale di Alta Matematica (INdAM).
The authors wish to thank the anonymous referee for several suggestions and remarks.


\end{document}